\documentclass[11pt,a4paper]{amsart}
\usepackage{amssymb,amsmath,amsthm,amsfonts}
\usepackage{mathrsfs}
\usepackage[mathscr]{euscript}
\usepackage{mathtools}
\usepackage{booktabs}
\usepackage{tikz}
\usepackage{tikz-cd}
\usetikzlibrary{matrix,decorations.pathmorphing,fit,shapes.geometric,backgrounds}
\usepackage[english]{babel}
\usepackage{newlfont}
\usepackage{lmodern} 
\usepackage{graphicx}
\usepackage{float}
\usepackage[margin=2.2cm]{geometry} 

\usepackage{etoolbox}
\usepackage{upgreek}
\usepackage[utf8]{inputenc}
\usepackage{hyperref}
\usepackage{thmtools,thm-restate}
\usepackage{float}
\usepackage{cleveref}
\usepackage{subcaption}
\usepackage{graphicx}
\usepackage{enumitem}
\usepackage{amsaddr}

\newcommand{\C}{\mathcal}
\newcommand{\s}{\mathsf}
\newcommand{\f}{\mathfrak}
\newcommand{\R}{\mathrm}
\newcommand{\B}{\mathbf}

\newcommand{\hei}[1]{\mathfrak{ht_{\mathrm{#1}}}}

\newcommand{\HT}{\hei{\C P}}
\newcommand{\ang}[1]{\langle#1\rangle}

\newcommand{\Hom}{\R{Hom}_{\Lambda}(M_1,M_2)}

\newcommand{\Br}{\mathsf{Br}}
\newcommand{\im}[1][\iota]{\mathrm{Im}(#1)}

\newcounter{thmcount}
\newtheorem{defn}{Definition}[section]

\newtheorem{lem}[defn]{Lemma}

\newtheorem{prop}[defn]{Proposition}
\newtheorem*{prop*}{Proposition}
\newtheorem{thm}[defn]{Theorem}

\newtheorem{cor}[defn]{Corollary}

\newtheorem*{claim*}{Claim}

\theoremstyle{remark}

\newtheorem{exmp}[defn]{Example}
\AtEndEnvironment{exmp}{\hfill$\Diamond$}

\newtheorem{que}[defn]{Question}
\newtheorem{rmk}[defn]{Remark}

\theoremstyle{definition}

\numberwithin{equation}{section}
\author{Suraj Mishra$^{\dagger}$ and Amit Kuber$^{\ddagger}$}
\email{suraj.mishra.4591@gmail.com,askuber@iitk.ac.in}
\address{$^\dagger$Department of Mathematics, Indian Institute of Technology, Bombay\\$^\ddagger$Department of Mathematics and Statistics, Indian Institute of Technology, Kanpur}
\title{Rooted tree modules}

\date{}


\begin{document}

\begin{abstract}
A \emph{rooted tree module} (RTM) $M:=M(T,F)$ over a zero-relation algebra $\Lambda:=\mathcal KQ/\langle\rho\rangle$ over a field $\mathcal K$ is given by the data of a quiver morphism $F:T\to Q$ from a rooted tree $T$ (either with a source or a sink) taking paths in $T$ to paths in $Q$ not lying in $\langle\rho\rangle$. When $\mathrm{char}(\mathcal K)\neq2$, we provide a checkable combinatorial characterization of the indecomposability of the RTM $M$ in terms of non-existence of idempotent quiver morphisms $\iota:T\to T$ satisfying $F\circ\iota=F$ and $\iota\neq 1_T$. Further, we provide an iterative method to decompose an RTM into indecomposable RTMs as well as a method to recursively construct indecomposable RTMs.
\end{abstract}

\maketitle

\section{Introduction}\label{intro}
Fix a field $\C K$ with $\mathrm{char}(\C K)\neq2$. A \emph{zero-relation algebra} $\Lambda$ is the quotient of the path algebra $\C KQ$ of a locally bound quiver $Q=(Q_0,Q_1,\varsigma,\varepsilon)$ by the ideal generated by a set $\rho$ of paths in $Q$ with length at least $2$ that gives an absolute bound on the length of a path passing through any vertex. In other words, $\Lambda\coloneqq\C KQ/\ang{\rho}$. There is an equivalence between the category $\Lambda\text{-}\mathrm{mod}$ of finite dimensional $\Lambda$-modules and the category of finite-dimensional $\C K$-representations of $(Q,\rho)$.

Let us introduce some notations associated with rooted trees before introducing the titular objects, namely rooted tree modules over $\Lambda$.
\begin{defn}
\label{defn: tree}
A \emph{tree} is a finite quiver $T=(T_0,T_1,s,t)$ whose underlying undirected graph is simply connected. Assume that $T_0\subset\mathbb N$. The tree $T$ is said to be a \emph{rooted tree} if it has a unique sink or a unique source, say $\ast\in T_0$, called the \emph{root} of the tree $T$.
\end{defn}

\begin{defn}
\label{defn: quiver morphism}
A \emph{quiver morphism} $F:T\to Q$ is a pair of functions $F_j:T_j\xrightarrow{}Q_j$ for $j=0,1$ such that $F_0\circ s=\varsigma \circ F_1$ and $F_0\circ t=\varepsilon\circ F_1$. We say that $F$ is a \emph{bound quiver morphism} if there is no path in $T$ such that $F(p)\in \ang{\rho}$.
\end{defn}
We write a bound quiver morphism as $F:T\to(Q,\rho)$, and say that $(T,F)$ is a \emph{tree over} the locally bound quiver $(Q,\rho)$ Such a bound quiver morphism induces a push-down functor $F_\lambda:\C KT\text{-}\mathrm{mod}\to\Lambda\text{-}\mathrm{mod}$ between the module categories as follows:
$$(F_\lambda((U_n)_{n\in T_0},(\varphi_a)_{a\in T_1}))_j\coloneqq\bigoplus_{n\in F^{-1}(j)}U_n,\quad (F_\lambda((U_n)_{n\in T_0},(\varphi_a)_{a\in T_1}))_\gamma\coloneqq\sum_{a\in F^{-1}(\gamma)}\varphi_a.$$

Associated to a tree $T$, there is a $|T_0|$-dimensional module $V_T\in\C KT\text{-}\mathrm{mod}$ given by the representation where all vertices of $T$ are replaced by a copy of the field $\C K$ and all arrows of $T$ are replaced by a copy of the identity morphism $1_{\C K}$.

\begin{defn}\cite[Definition~1.6]{sengupta2025generalised}
The \emph{generalized tree module} (\emph{GTM}) corresponding to a tree $(T,F)$ over $(Q,\rho)$ is the module $M(T,F)\coloneqq F_\lambda(V_T)$. If $(T,\ast)$ is a rooted tree then, by an abuse of terminology, we will refer to $M(T,F)$ as a \emph{rooted tree module} (\emph{RTM}) instead of a generalized rooted tree module.
\end{defn}

Suppose $M_1\coloneqq M(T^1,F_1)$ and $M_2\coloneqq M (T^2,F_2)$ are GTMs. The module $M_j$ is called a \emph{tree module} if $F_j$ satisfies an additional property: $$\text{ for all } a,b\in T^j_1,\text{ if } (s(a)=s(b)\text{ or }t(a)=t(b)) \text{ then } F_j(a)\neq F_j(b) .$$ It follows from \cite[\S~3.5,4.1]{gabriel1981covering} that tree modules over a zero-relation algebra are indecomposable. Crawley-Boevey \cite{crawley1989maps} gave a combinatorial description of a basis for $\Hom$ when $M_1$ and $M_2$ are tree modules. Ringel \cite{ringelexceptional} proved that all the finitely generated indecomposable \emph{exceptional modules}(=modules without self-extensions) over path algebras of quivers are GTMs; however, he referred to them as tree modules. When $M_1,M_2$ are GTMs, generalizing \cite{crawley1989maps}, Sengupta and the second author \cite{sengupta2025generalised} introduced the combinatorial concept of a generalized graph map (GGM) for the pair $(M_1,M_2)$ (\Cref{defn:ggm}), and showed, under some technical condition which they call ``ghost-freeness'' of the pair (\Cref{defn:ghost}), the finite set of GGMs spans $\Hom$ \cite[Theorem~A]{sengupta2025generalised} as a $\C K$-space.

Now suppose $T^1$ and $T^2$ be rooted trees, either both with sinks or both with sources. For the corresponding RTMs $M_1$ and $M_2$, we show in \Cref{thm: ghostfree} (sink case) and \Cref{thm: ghostfreeop} (source case) that the pair $(M_1,M_2)$ is ghost-free. These lemmas along with a sufficient condition for indecomposability of a GTM \cite[Theorem~B]{sengupta2025generalised} are some of the key ingredients in the proof of our main result, which we state now.
\begin{restatable}{theom}{mainone}\label{thm: main A}
Suppose $(T,F)$ is a rooted tree over $(Q,\rho)$. Then the following are equivalent:
\begin{enumerate}
    \item the RTM $M\coloneqq M(T,F)$ is indecomposable;
    \item there is no non-identity idempotent quiver morphism $\iota: T\to T$ satisfying $F\circ\iota=F$.
\end{enumerate}
\end{restatable}
This theorem provides a checkable combinatorial criterion for testing the indecomposability of an RTM. Compare this result with the following characterization of the indecomposability of a finitely generated module.
\begin{prop}
\cite[Lemma~I.4.6, Corollary~I.4.8]{assem2006elements}\label{prop: indecomposability condn assem}
If $M$ is a finite-dimensional $\Lambda$-module, then $M$ is indecomposable if and only if the endomorphism algebra $\R{End}_\Lambda(M)$ contains only two idempotents, viz. $\B 0$ and $\B 1_M$.
\end{prop}

As an application of our main result, we provide a recursive construction of indecomposable RTMs \Cref{cor: 2} (sink case) and \Cref{cor: 2op} (source case). We also show how to decompose an RTM into indecomposable direct summands \Cref{cor: 1} (sink case) and \Cref{cor: 1op} (source case). A special case of this corollary when both $T$ and $Q$ are rooted trees with sinks and $\rho=\emptyset$ \cite[Lemma~2]{katter_mahrt} was proven by Katter and Mahrt. Under the same hypotheses, Bindua et al. \cite{bindua2024} proved an interesting result about decompositions of modules in the essential image of the zero-dimensional persistence homology functor.

The rest of the paper is organized as follows. The sink case is dealt with in detail in \S~\ref{sec 2} and \ref{sec 3}. We set up notations and terminology related to generalized graph maps in \S~\ref{sec 2}. The main highlight of this section is \Cref{GGMgeneration} which shows that the Hom-set between two rooted tree modules is spanned by generalised graph maps. We complete the proof of \Cref{thm: main A} in the sink case in \S~\ref{sec 3} and prove applications of the the theory developed to decompositions of RTMs (\Cref{cor: 1}) and recursive construction of indecomposable RTMs (\Cref{cor: 2}). The changes in the definitions, statements and proofs of results for the source case are given in \S~\ref{sec 4}.

The notation $\mathbb N$ denotes the set of natural numbers that includes $0$.

\section{Spanning sets for Hom-sets between RTMs with sink}\label{sec 2}
In this section, the notation $(T,\ast)$ (possibly with decoration) will denote a rooted tree with a sink $\ast$. We recall the notations and terminology of GGMs from \cite{sengupta2025generalised} in the context of RTMs with sinks. The main goal is to prove \Cref{GGMgeneration} which shows that the finite set of GGMs is a $\C K$-spanning set for the Hom-set between two RTMs with sinks.

For $n\in T_0$, the \emph{branch} of $n$, denoted $\Br(n)=(\Br(n)_0,\Br(n)_1)$ is a subquiver of the tree $T$ with $\Br(n)_0\coloneqq\{n'\in T_0\mid\text{there is a path from } n' \text{ to } n\}$ and $\Br(n)_1\coloneqq \{a_m\mid m\in \Br(n)_0\setminus\{n\}\}$. Note that the existence of zero-length paths implies $n\in \Br(n)$ for each $n\in T_0$ so that $(\Br(n),n)$ is itself a rooted tree. If $X\subseteq T_0$, then we use the notation $\ang{X}$ to denote the subforest of $T_0$ induced by the subset $X$.

Let $(T^j,\ast_j)$ be a rooted tree for $j=1,2$. If $n\in T^j_0\setminus\{\ast_j\}$, then there is a unique arrow with source $n$; we denote by $\s p(n)$ the \emph{parent} of $n$, i.e., the target of such unique arrow. Let $a_n$ (resp. $b_n$) denote the unique arrow with source $n$ if $n\in T^1_0\setminus\{\ast_1\}$ (resp. if $n\in T^2_0\setminus\{\ast_2\}$). Define the \emph{height} function $\hei{}\colon T^j\to \mathbb{N}$ as
$\hei{}(n)\coloneqq\begin{cases}
        0 & \text{if } n=\ast_j;\\
        \hei{}(\s p(n))+1 & \text{otherwise.}
    \end{cases}$ Also set $\hei{}(T^j)\coloneqq\max\{\hei{}(n)\mid n\in T^j_0\}$.
\begin{rmk}
    For each $n\in T_0^j$, $\s p^{\hei{}(n)}(n)=\ast_j$.
\end{rmk}
Since quiver morphisms between rooted trees preserve lengths of paths, we have the following observation.
\begin{rmk}\label{rem:heightpreserve}
Suppose $\iota:\Br(n)\to\Br(m)$ is a quiver morphism for some $n\in T^1_0$ and $m\in T^2_0$. Then $\hei{}(\Br(n))\le\hei{}(\Br(m))$. 
\end{rmk}

Let $F_j:T^j\to(Q,\rho)$ be a bound quiver morphism, and $M_j\coloneqq M(T^j,F_j)$ be the corresponding RTM for $j=1,2$. Let $\{v_n\}_{n\in T_0^1}$ and $\{w_m\}_{m\in T_0^2}$ denote the natural $\C K$-bases of $M_1$ and $M_2$ respectively that are induced by the quiver morphisms.

The definition of a GGM uses the language of networks(=mixed graphs), which we define below.
\begin{defn}\cite[Definition~2.1]{sengupta2025generalised}
    A \emph{network} $\C N$ is defined as a pentuple $(\C N_0,\C N_1,\sigma,\tau,\C E)$, where $(\C N_0,\C N_1,\sigma,\tau)$ is a quiver and $(\C N_0,\C E)$ is a simple undirected graph.
\end{defn}
For $\C V\subseteq \C N_0$, the notation $\ang{\C V}$ denotes the subnetwork of $\C N$ induced by $\C V$.

Associated with the pair $(M_1,M_2)$ of RTMs are two natural networks--the first one is called the pullback network, for reasons explained later, while the other is a $2$-cover of the former.

\begin{defn}\cite[\S~2]{sengupta2025generalised}\label{pullback}
The \emph{pullback network} $\C N[1]\coloneqq (\C N[1]_0,\C N[1]_1,s^1,t^1,\C E^1)$ associated with the pair $(M_1,M_2)$ is defined as follows:
\begin{align*}
    \C N[1]_0&\coloneqq\{(n,m)\in T^1_0\times T^2_0\ |\ F_1(n)=F_2(m)\}; \ \mathrm{and}\\
    \C N[1]_1&\coloneqq\{(n,m)\xrightarrow{(a_n,b_m)}(\s p(n),\s p(m))\mid F_1(a_n)=F_2(b_m)\};\\
    \C E^1&\coloneqq\{\{(n,m),(n,m')\}\mid \s p(m)=\s p(m'), F_2(b_m)=F_2(b_{m'})\}.
\end{align*}
    
\end{defn}
The quadruple $\C P\coloneqq(\C N[1]_0,\C N[1]_1,s^1,t^1) $ will be referred to as the \emph{pullback quiver} and it is indeed a pullback as shown in \Cref{pullbacknetwork}.
\begin{figure}[H]
\centering
\begin{tikzcd}
    \C P & T^1\\
    T^2 & Q
    \arrow["\pi_1", from=1-1, to=1-2]
    \arrow["\pi_2"', from=1-1, to=2-1]
    \arrow["F_1", from=1-2, to=2-2]
    \arrow["F_2"', from=2-1, to=2-2]
    \arrow["\lrcorner"{anchor=center, pos=0.001}, draw=none, from=1-1, to=2-2]
\end{tikzcd}
\caption{Pullback quiver associated with the pair $(M_1,M_2)$}
\label{pullbacknetwork}
\end{figure}
For $(n,m)\in\C N[1]_0$, set $\hei{\C P}((n,m))\coloneqq\max\{k\mid(\s p^{k'}(n),\s p^{k'}(m))\in \C N[1]_0\text{ for all } 0\le k'\le k \}$. By an abuse of notation, we will omit the extra pair of parentheses and write $\HT ((n,m))$ as $\HT (n,m)$. If $\hei{\C P}(n,m)>0$, then define its \emph{parent} to be $\s p(n,m)\coloneqq(\s p(n),\s p(m))$.
\begin{rmk}\label{isasink}
If $(n,m)\in\C N[1]_0$, then $\s p^{\HT(n,m)}(n,m)$ is a sink in the pullback quiver $\C P$.
\end{rmk}
\begin{rmk}\label{rmk: outdegree}
    The out-degree of $(n,m)$ in $\C P$ is at most $1$, and it is $0$ if and only if $\HT(n,m)=0$.
\end{rmk}


The next observation will be crucial to determine all ``admissible walks'' in $\C N[1]$.
\begin{rmk}\label{edgetransitive}
If $\{(n,m_1),(n,m_2)\}, \{(n,m_2),(n,m_3)\}\in \C E^1$ then $\s p(m_1)=\s p(m_2)=\s p(m_3)$, and hence $\{(n,m_1),(n,m_3)\}\in \C E^1$.
\end{rmk}

\begin{exmp}\label{exmp: RTM}
Consider an RTM $M\coloneqq M(T,F)$ over the bound quiver $\C KQ/\ang{\rho}$ whose data is shown in \Cref{fig: rooted tree}: the rooted tree $(T,1)$ is shown in \Cref{subfig: tree}, the bound quiver $(Q,\rho)$ is shown in \Cref{subfig: quiver} while a bound quiver morphism $F:T\to(Q,\rho)$ is given in the caption of the figure.
\begin{figure}[H]
    \centering
    \begin{subfigure}[t]{0.51\textwidth}
        \[\begin{tikzcd}[ampersand replacement=\&,cramped]
	5 \\
	2 \& 3 \& 4 \\
	\& 1
	\arrow["{a_5}"{description}, from=1-1, to=2-1]
	\arrow["{a_2}"{description}, from=2-1, to=3-2]
	\arrow["{a_3}"{description}, from=2-2, to=3-2]
	\arrow["{a_4}"{description}, from=2-3, to=3-2]
    \end{tikzcd}\]
    \subcaption[]{A rooted tree $(T,1)$}
    \label{subfig: tree}
    \end{subfigure}~
    \begin{subfigure}[t]{0.3\textwidth}
        \[\begin{tikzcd}[ampersand replacement=\&,cramped]
	{\boldsymbol{1}} \\
	{\boldsymbol{2}}
	\arrow["\beta", from=1-1, to=2-1]
	\arrow["\alpha", from=2-1, to=2-1, loop, in=235, out=305, distance=12mm]
\end{tikzcd}\]
    \subcaption[]{Bound quiver with $\rho=\{\alpha^2\}$}
    \label{subfig: quiver}
    \end{subfigure}
\caption{Bound quiver morphism $F\colon T\to (Q,\rho)$ given by $F(3)=F(5)=\boldsymbol{1};\\ F(1)=F(2)=F(4)=\boldsymbol{2}; F(a_3)=F(a_5)=\beta; F(a_2)=F(a_4)=\alpha$}
\label{fig: rooted tree}
\end{figure}
The pullback network $\C N[1]$ associated with the pair $(M,M)$ is shown in \Cref{fig: network rooted tree}.
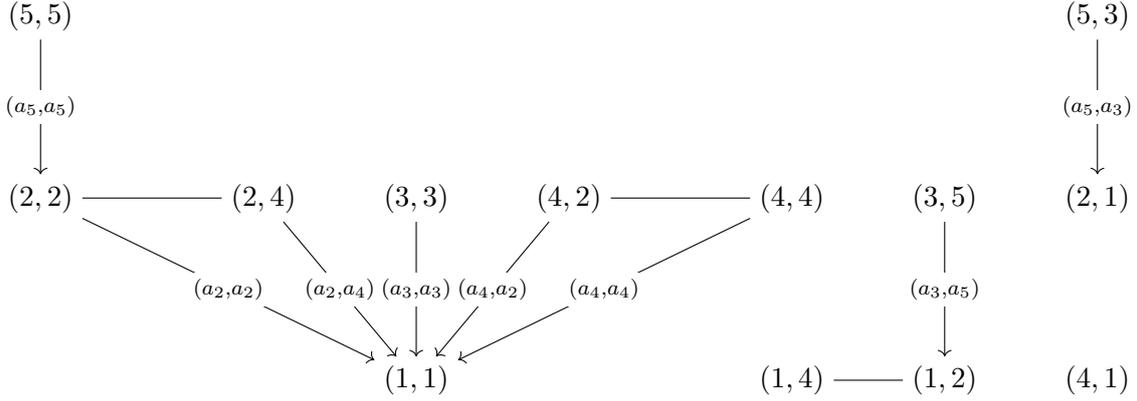
\begin{figure}[H]
    \centering
   \[\begin{tikzcd}[ampersand replacement=\&,cramped,row sep=9mm]
	{(5,5)} \&\&\&\&\&\&\&\& {(5,3)} \\
	\\
	{(2,2)} \&\& {(2,4)} \& {(3,3)} \& {(4,2)} \&\& {(4,4)} \& {(3,5)} \& {(2,1)} \\
	\\
	\&\&\& {(1,1)} \&\&\& {(1,4)} \& {(1,2)} \& {(4,1)}
	\arrow["{(a_5,a_5)}"{description}, from=1-1, to=3-1]
	\arrow["{(a_5,a_3)}"{description}, from=1-9, to=3-9]
	\arrow[no head, from=3-1, to=3-3]
	\arrow["{(a_2,a_2)}"{description}, from=3-1, to=5-4]
	\arrow["{(a_2,a_4)}"{description}, from=3-3, to=5-4]
	\arrow["{(a_3,a_3)}"{description}, from=3-4, to=5-4]
	\arrow["{(a_4,a_2)}"{description}, from=3-5, to=5-4]
	\arrow[no head, from=3-7, to=3-5]
	\arrow["{(a_4,a_4)}"{description}, from=3-7, to=5-4]
	\arrow["{(a_3,a_5)}"{description}, from=3-8, to=5-8]
	\arrow[no head, from=5-8, to=5-7]
\end{tikzcd}\]
    \caption{The network $\C N[1]$ associated with the pair $(M,M)$ from \Cref{exmp: RTM}}
    \label{fig: network rooted tree}
\end{figure}
\end{exmp}
\begin{defn}\cite[Definition~2.11]{sengupta2025generalised}
    The $2$-covering network $\C N[2]\coloneqq(\C N[2]_0,\C N[2]_1,s^2, t^2,\C E^2)$ associated with the pair $(M_1,M_2)$ is defined as follows:
    \begin{align*}
        \C N[2]_l&\coloneqq\C N[1]_l\times \{-1,1\} \mathrm{\ for\ }l=0,1,\\
        s^2(a,b,i)&\coloneqq(s^1(a,b),i) \mathrm{\ for\ } (a,b,i)\in \C N[2]_1,\\
        t^2(a,b,i)&\coloneqq(t^1(a,b),i) \mathrm{\ for\ } (a,b,i)\in \C N[2]_1,\\
        \C E^2&\coloneqq\{\{(n,m,i),(n,m',-i)\}\mid \{(n,m),(n,m')\}\in \C E^1, i\in\{-1,1\}\}.
    \end{align*}    
\end{defn}
\begin{rmk}\cite[Remark~2.12]{sengupta2025generalised}
    If $\mathrm{char}(\C K)=2$, then $\C N[1]=\C N[2]$ and hence we assumed that $\mathrm{char}(\C K)\neq2$.
\end{rmk}
\begin{rmk}\cite[Remark~2.13]{sengupta2025generalised}
    There is a canonical projection $\pi\colon\C N[2]\to \C N[1]$ given by
    \begin{eqnarray*}
        \pi((x,y,i))\coloneqq (x,y) \text{ for } (x,y,i)\in \C N[2]_0\cup\C N[2]_1
    \end{eqnarray*}
    satisfying $\pi\circ s^2=s^1\circ \pi$, $\pi\circ t^2=t^1\circ \pi$, and mapping $\C E^2$ to $\C E^1$ surjectively. Since $| \pi^{-1}(\alpha)|=2$ for $\alpha \in \C N[1]_0\cup\C N[1]_1\cup\C E^1$ is indeed a $2$-cover of $\C N[1].$
\end{rmk}

A \emph{link} in $\C N[j]$ is an element of $\C N[j]_1\sqcup\C N[j]^{-1}_1\sqcup \C E^j$, where $\C N[j]_1^{-1}$ is the set of symbols $\alpha^{-1}$ as $\alpha$ varies over $\C N[1]_1$ and the intended meaning $\alpha^{-1}$ is that of the reverse along arrow $\alpha$. We denote a link using the notation ``$--$". A \emph{traversal} $\f t$ of length $k>0$ in $\C N[j]$ is a walk of length $k$ along links that does not contain $\alpha \alpha^{-1}$ as a subtraversal for any link $\alpha$ with the convention that $(\alpha^{-1})^{-1}=\alpha$ if $\alpha$ is an arrow, and $\alpha^{-1}=\alpha$ if $\alpha$ is an edge; we treat vertices of $\C N[j]$ as zero-length traversals. The reader is referred to \cite[Definition~2.2]{sengupta2025generalised} for the complete technical definition of a traversal. Let $\f T(\C N[j])$ denote the set of all traversals in $\C N[j]$. We use the convention that a traversal is read from right to left, i.e., if $\f t=\alpha_k\cdots\alpha_1$ then the end of $\alpha_i$ is the same as the source of $\alpha_{i+1}$. Given $\f t\in\f T(\C N[j])$, we will denote its \textit{length}(=number of links contained within) by $|\f t|$ and its inverse(=walk in reverse direction) by $\f t^{-1}$. An interested reader can refer to \cite[Remark~2.3]{sengupta2025generalised} for a complete treatment of the notation.
\begin{rmk}\cite[Remark~2.5]{sengupta2025generalised}\label{rmk: at most one link}
Since $T^1$ and $T^2$ are trees, there is at most one link between two distinct vertices of $\C N[1]$, and there is no directed cycle of arrows in $\C N[1]$.
\end{rmk}

Since a finite directed graph where the out-degree of each vertex is at most $1$ is a forest of rooted trees if and only if it does not contain any directed cycle, Remarks \ref{rmk: outdegree} and \ref{rmk: at most one link} together yield the following observation.
\begin{rmk}\label{forest}
The pullback quiver $\C P$ is a forest with each connected component a rooted tree.
\end{rmk} 
In view of the above remark, the definition of the branch of a vertex could be extended to $\C P$.
\begin{exmp}
Continuing from \Cref{exmp: RTM}, observe that the pullback quiver associated with the pair $(M,M)$ is a forest with five rooted trees with roots $(1,1),(1,4),(1,2),(4,1)$ and $(2,1)$.
\end{exmp}

Associated with a subnetwork $\C M=(\C M_0,\C M_1,\C E_{\C M})$ of $\C N[2]$, there is a linear map $\C H_{\C M}:M_1\to M_2$ defined by $\C H_{\C M}(v_n)\coloneqq\sum_{(n,m,j)\in \C M_0}j w_m$. We need $\C M$ to satisfy two conditions, namely completeness and $\C R[2]$-freeness, which ensure that $\C H_{\C M}$ is a $\Lambda$-module homomorphism.
\begin{defn}\label{defn: completeness}\cite[Definition~3.2]{sengupta2025generalised}
A subnetwork $\C M=(\C M_0,\C M_1,\C E_{\C M})$ of $\C N[2]$ is said to be \emph{complete} if for each $(n,m,j)\in \C M_0$, we have the following conditions.
\begin{enumerate}
    \item If $(n'\xrightarrow{a_{n'}}n)\in T^1_1$, then there exists $(m'\xrightarrow{b_{m'}}m)\in T^2_1$ such that $$((n',m',j)\xrightarrow{(a_{n'},b_{m'},j)}(n,m,j))\in \C M_1.$$
    \item If $(m\xrightarrow{b_m}m')\in T^2_1$, then at least one of the following holds:
    \begin{enumerate}
        \item there exists $(n\xrightarrow{a_n}n')\in T^1_1$ such that $((n,m,j)\xrightarrow{(a_n,b_m,j)}(n',m',j))\in \C M_1$;
        \item there exists $(m\xrightarrow{b_m}m'\xleftarrow{b_{m''}}m'')$ in $T^2$ with $F_2(b_m)=F_2(b_{m''})$ and $(n,m'',-j)\in \C M_0$ such that $$\{(n,m,j),(n,m'',-j)\}\in \C E_{\C M}.$$
    \end{enumerate}
\end{enumerate}
\end{defn}
A \emph{triangle} in $\C N[1]$ \cite[\S~2]{sengupta2025generalised} is a subnetwork $\ang{\C V}$ such that $|\ang{\C V}|_0=|\ang{\C V}|_1+|\C E_{\ang{\C V}}|=3$. We denote the set of all triangles in $\C N[1]$ by $\triangle$. There are only two type of triangles in $\C N[1]$ as shown in \Cref{triatype}.
\begin{figure}[H]
    \centering
    \begin{subfigure}[b]{0.45\textwidth}
        \begin{tikzcd}[column sep=5mm]
            (n,m)& & (n,m')\\
            & (p(n),p(m))
            \arrow["{(a_n,b_{m})}"',from=1-1,to=2-2]
            \arrow["{(a_n,b_{m'})}",from=1-3,to=2-2]
            \arrow[no head, from=1-1,to=1-3]
        \end{tikzcd}
        \caption{Triangle with $1$ edge}
        \label{arrtria}
    \end{subfigure}%
    ~
    \begin{subfigure}[b]{0.45\textwidth}
        \begin{tikzcd}[column sep=5mm]
            (n,m)& & (n,m')\\
            & (n,m'')
            \arrow[no head,from=1-1,to=2-2]
            \arrow[no head,from=1-3,to=2-2]
            \arrow[no head, from=1-1,to=1-3]
        \end{tikzcd}
        \caption{Triangle with $3$ edges}
        \label{noarrtria}
    \end{subfigure}
    \caption{Triangles in $\C N[1]$}
    \label{triatype}
\end{figure}

In order to ensure uniqueness in each condition in the definition above, we need to ``block'' some traversals in $\C N[j]$:
\begin{align*}
\C R[1]&\coloneqq \{(n_1,m_1)--(n_2,m_2)--(n_3,m_3)\in \f T(\C N[1])\mid \ang{\{(n_k,m_k)\mid k\in\{1,2,3\}\}}\in \triangle\};\text{ and}\\
\C R[2]&\coloneqq \{\f t\in \f T(\C N[2])\mid \pi(\f t)\in \C R[1]\}.
\end{align*}
\begin{defn}\cite[Definition~3.5]{sengupta2025generalised}
For $j=\{1,2\}$, a subnetwork $\C M$ of $\C N[j]$ is said to be \emph{$\C R[j]$-free} if there is no traversal in $\C M$ that lies in $\C R[j].$
\end{defn}

Here is the promised result about certain subnetworks yielding homomorphisms.
\begin{prop}\cite[Proposition~3.9]{sengupta2025generalised}
If a subnetwork $\C M$ of $\C N[2]$ is complete and $\C R[2]$-free, then the associated linear map $\C H_{\C M}:M_1\to M_2$ is a $\Lambda$-module homomorphism.
\end{prop}

Now we prove a key technical result that characterizes all $\C R[j]$-free traversals in $\C N[j]$.
\begin{prop}\label{key}
Suppose $\f t\coloneqq\alpha_k\cdots\alpha_1\in\f T(\C N[j])$ is $\C R[j]$-free for $j\in\{1,2\}$ and $k>1$. Then the following are true.
\begin{enumerate}
    \item If $1\leq i<k$ and $\alpha_i\notin\C N[j]_1$, then $\alpha_{i+1}\in\C N[j]_1^{-1}$.
    \item If $1<i\leq k$ and $\alpha_i\notin\C N[j]_1^{-1}$, then $\alpha_{i-1}\in\C N[j]_1$.
\end{enumerate}
\end{prop}
\begin{proof}
We only prove the first statement; the second statement can be proven by applying the first to the traversal $\f t^{-1}$ for the index $k-i+1$. Furthermore, note that it is enough to prove the result for $j=2$; the proof for $j=1$ can be obtained from it by applying the projection $\pi$.

There are two cases.

\noindent{\textbf{Case 1:}} Suppose $\alpha_i\coloneqq\{(n,m_1,l),(n,m_2,-l)\}\in\C E^2$ for some $l\in\{-1,1\}$.

If $\alpha_{i+1}\in\C E^2$, then \Cref{edgetransitive} guarantees that $\alpha_{i+1}\alpha_i\in\C R[2]$, a contradiction. Thus $\alpha_{i+1}\notin\C E^2$. Thus, without loss of generality, we may assume that $s^2(\alpha_{i+1})=(n,m_2,-l)$.

If $\alpha_{i+1}\in\C N[2]_1$, then \Cref{rmk: outdegree} ensures that $(n,m_2,-l)\in\C N[2]_0$ and $\alpha_{i+1}=(a_n,b_{m_2},-l)$. But since $\alpha_{i+1}\alpha_i\in\C R[2]$, we get a contradiction yet again, thus completing the proof for this case.

\noindent{\textbf{Case 2:}} Suppose $\alpha_i^{-1}=(a_n,b_m,l)\in\C N[2]_1$ for some $l\in\{-1,1\}$.

If $\alpha_{i+1}\in\C N[2]_1$, then \Cref{rmk: outdegree} ensures that $\alpha_{i+1}=\alpha_i^{-1}$, which is a contradiction to the definition of a traversal.

If $\alpha_{i+1}\in\C E^2$, then $\alpha_{i+1}=\{(n,m,l),(n,m',-l)\}$ for some $m'\in T^2_0$. But then $\alpha_{i+1}\alpha_i\in\C R[2]$, a contradiction that completes the proof of this case as well as that of the proposition.
\end{proof}

As an immediate consequence of the above, the number of edges in an $\C R[j]$-free traversal could be bounded above.
\begin{cor}\label{only1edge}
If $\f t\in\f T(\C N[j])$ is $\C R[j]$-free for some $j\in\{1,2\}$, then $\f t$ contains at most one edge.
\end{cor}
\begin{proof}
We prove the result only for $j=1$; the other case has analogous proof.

Let $k\coloneqq|\f t|$. If $k\leq 1$, then the conclusion is obvious. Hence assume $\f t=\alpha_k\cdots\alpha_1$ for $k>1$, and $\alpha_q=\{(n,m_1),(n,m_2)\}\in\C E^1$ for some $1\leq q\leq k$. Since $\alpha_q\in\C E^1$, we have $\s p(m_1)=\s p(m_2)$.

If $q<k$, then repeated applications of \Cref{key}(1) yield that $\alpha_i\in\C N[1]^{-1}$ for each $q<i\leq k$.

Similarly, if $q>1$, then repeated applications of \Cref{key}(2) yield that $\alpha_i\in\C N[1]$ for each $1\leq i<q$. This completes the proof. 
\end{proof}

\Cref{fig: R1freetraversals} shows the sketches of all $\C R[1]$-free traversals in $\C N[1]$.

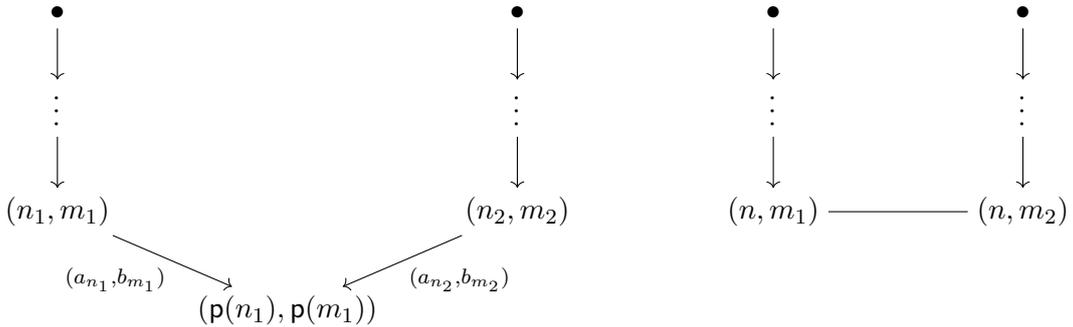
\begin{figure}[H]
    \centering
        \[\begin{tikzcd}[ampersand replacement=\&,cramped]
	\bullet \&\& \bullet \&\& \bullet \&\& \bullet \\
    \rotatebox{90}{$\cdots$} \&\& \rotatebox{90}{$\cdots$} \&\& \rotatebox{90}{$\cdots$} \&\& \rotatebox{90}{$\cdots$} \\
	{(n_1,m_1)} \&\& {(n_2,m_2)} \&\& (n,m_1)\&\& (n,m_2)\\
	\& {(\s p(n_1),\s p(m_1))}
	\arrow[from=1-1, to=2-1]
	\arrow[from=1-3, to=2-3]
	\arrow[from=1-5, to=2-5]
	\arrow[from=1-7, to=2-7]
	\arrow[from=2-1, to=3-1]
	\arrow[from=2-3, to=3-3]
	\arrow["{(a_{n_1},b_{m_1})}"', from=3-1, to=4-2]
	\arrow["{(a_{n_2},b_{m_2})}", from=3-3, to=4-2]
    \arrow[from=2-5,to=3-5]
    \arrow[from=2-7,to=3-7]
	\arrow[no head, from=3-5, to=3-7]
\end{tikzcd}\]
    \caption{Sketches of $\C R[1]$-free traversals in $\C N[1]$: (L) w/o an edge when $n_1\neq n_2$, (R) with an edge}
    \label{fig: R1freetraversals}
\end{figure}
\begin{exmp}\label{exp:R1free}
Continuing from \Cref{exmp: RTM}, there are thirteen maximal $\C R[1]$-free traversals in $\C N[1]$--three of them are the smaller connected components while the remaining ten are included within the largest connected component. Two $\C R[1]$-free traversals included within the largest component are shown in \Cref{fig: R1free exmp}.
\end{exmp}
    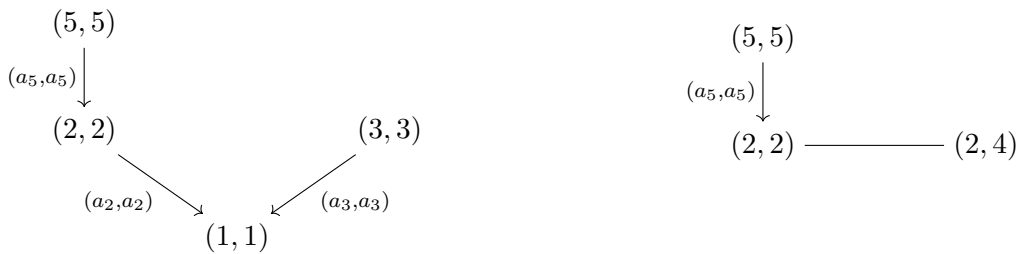
\begin{figure}[H]
        \centering
        \begin{subfigure}{0.495\textwidth}
            \[\begin{tikzcd}[ampersand replacement=\&,cramped,row sep=8mm]
	{(5,5)} \\
	{(2,2)} \&\& {(3,3)} \\
	\& {(1,1)}
	\arrow["{(a_5,a_5)}"', from=1-1, to=2-1]
	\arrow["{(a_2,a_2)}"', from=2-1, to=3-2]
	\arrow["{(a_3,a_3)}", from=2-3, to=3-2]
\end{tikzcd}\]
        \end{subfigure}~
        \begin{subfigure}{0.495\textwidth}
            \[\begin{tikzcd}[ampersand replacement=\&,cramped,row sep=8mm]
	{(5,5)} \\
	{(2,2)} \&\& {(2,4)} \\
	\arrow["{{(a_5,a_5)}}"', from=1-1, to=2-1]
    \arrow[no head, from=2-1,to=2-3]
\end{tikzcd}\]
        \end{subfigure}
        \caption{Two $\C R[1]$-free traversal in the network $\C N[1]$ from \Cref{exmp: RTM}}
        \label{fig: R1free exmp}
    \end{figure}
Say that a subnetwork $\C M$ of $\C N[2]$ is \emph{involution-free} if $\C M_0\cap\{(n,m,-j)\mid (n,m,j)\in \C M_0\}=\emptyset$, and \emph{connected} if there is a traversal between any two of its vertices.
\begin{defn}\cite[Definition~3.10]{sengupta2025generalised}\label{defn:ggm}
A \emph{generalized graph map} (\emph{GGM}) $\C G$ from $M_1$ to $M_2$ is a non-empty complete, connected, $\C R[2]$-free and involution-free subnetwork of $\C N[2]$.    
\end{defn}

If $\C G$ is a GGM from $M_1$ to $M_2$, then so is $-\C G$, where third entries of all the elements in $\C G_0,\C G_1$ and $\C E_{\C G}$ have their signs reversed.

There could be some problematic subnetworks of $\C N[2]$ which satisfy completeness, $\C R[2]$-freeness and connectedness, and yet their associated homomorphisms are identically $\boldsymbol{0}$.
\begin{defn}\cite[Definition~5.1]{sengupta2025generalised}\label{defn:ghost}
A non-empty, connected, complete and $\C R[2]$-free subnetwork $\C M\coloneqq(\C M_0,\C M_1,\C E_{\C M})$ of $\C N[2]$ is called a \emph{ghost} if $\C H_{\C M}=\boldsymbol{0}$. Say that the pair $(M_1,M_2)$ of RTMs is \emph{ghost-free} if $\C N[2]$ does not include a ghost.
\end{defn}
\begin{rmk}\label{ghostproperties}
    A ghost $\C M=(\C M_0,\C M_1,\C E_{\C M})$  satisfies the following properties: 
    \begin{enumerate}
        \item (Involution-invariant) If $(n,m,j)\in \C M_0$ then $(n,m,-j)\in \C M_0$.
        \item (Non-empty edge set) $\C E_{\C M}\neq\emptyset$.
    \end{enumerate}
\end{rmk}
Now we are ready to state and prove the main results of this section.
\begin{lem}\label{thm: ghostfree}
The pair $(M_1,M_2)$ of RTMs is ghost-free.
\end{lem}
\begin{proof}
Suppose for the sake of contradiction that $\C M=(\C M_0,\C M_1,\C E_{\C M})$ is a ghost in $\C N[2]$. Then \Cref{ghostproperties}(2) yields $\C E_{\C M}\neq \emptyset$. Let $\beta\coloneqq\{(n,m,j),(n,m',-j)\}\in \C E_{\C M}$. Then \Cref{ghostproperties}(1) yields that $(n,m',j)\in\C M_0$ and the connectedness of $\C M$ ensures that there is a traversal from $(n,m,j)$ to $(n,m',j)$ in $\C M$.

Let $\f t\coloneqq\alpha_k\cdots\alpha_1$ be a traversal in $\C M$ with $k\geq 1$ and $s^2(\f t)=(n,m,j)$. Since $\C M$ is $\C R[2]$-free, so is $\f t$. We will show that $t^2(\f t)\neq(n,m',j)$.

Since $\beta\in\C E_{\C M}$, we have $\s p(m)=\s p(m')$. Thus, there is no arrow from $m$ to $m'$ in $T^2$. Hence, $k>1$.

Since $\beta\in\C E_{\C M}$, $\C M$ is $\C R[2]$-free and $(a_n,b_m,j)\beta\in\C R[2]$, we conclude $(a_n,b_m,j)\notin\C M_1$. In particular, $\alpha_1\neq(a_n,b_m,j)$. But then \Cref{rmk: outdegree} ensures that $\alpha_1\notin\C N[2]_1$. Therefore, all the hypotheses of \Cref{key}(1) are satisfied for $\f t$, which then yields that only two cases are possible.

Suppose $\alpha_1\in\C N[2]_1^{-1}$. Then repeated applications of \Cref{key}(1) yield that $\alpha_i\in\C N[2]_1^{-1}$ for each $i$. In other words, $\pi(t^2(\alpha_i))\in\Br(n,m)_0$ for each $i$. Since $(n,m')\notin\Br(n,m)_0$, we get $t^2(\f t)\neq(n,m',j)$.

Suppose $\alpha_1\in\C E^2$. If $k>1$, then repeated applications of \Cref{key}(1) yield that $\alpha_i\in\C N[2]_1^{-1}$ for each $i>1$. In other words, for each $i\geq1$ we have $t^2(\alpha_i)=(n'',m'',-j)$. Thus, clearly $t^2(\f t)\neq(n,m',j)$.
\end{proof}

\Cref{thm: ghostfree} together with \cite[Theorem~A]{sengupta2025generalised} yields the following.
\begin{cor}\label{GGMgeneration}
Suppose $(M_1,M_2)$ is a pair of RTMs with sinks. Then $\Hom$ is the $\C K$-span of the set $\{\C H_{\C G}\mid\C G\text{ is a GGM from }M_1\text{ to }M_2\}$.
\end{cor}
\begin{exmp}\label{exmp: all GGMs}
Continuing from \Cref{exp:R1free}, we note that there are only ten GGMs from the RTM $M$ (see \Cref{fig: rooted tree}) to itself--the complete set is $\{\pm\C G^i\mid 1\leq i\leq 5\}$, where all $\C G^i$ are shown in \Cref{fig: all GGMs} so that $\R{End}_{\Lambda}(M)=\mathrm{span}_{\C K}\{\C H_{\C G^i}\mid 1\leq i\leq 5\}$. Note that no GGM can contain $(2,4,j)$ and $(1,4,j)$ for any $j\in\{-1,1\}$ for the completeness condition at those points for the incoming arrows at $2$ and $1$ respectively are not satisfied.
\end{exmp}
    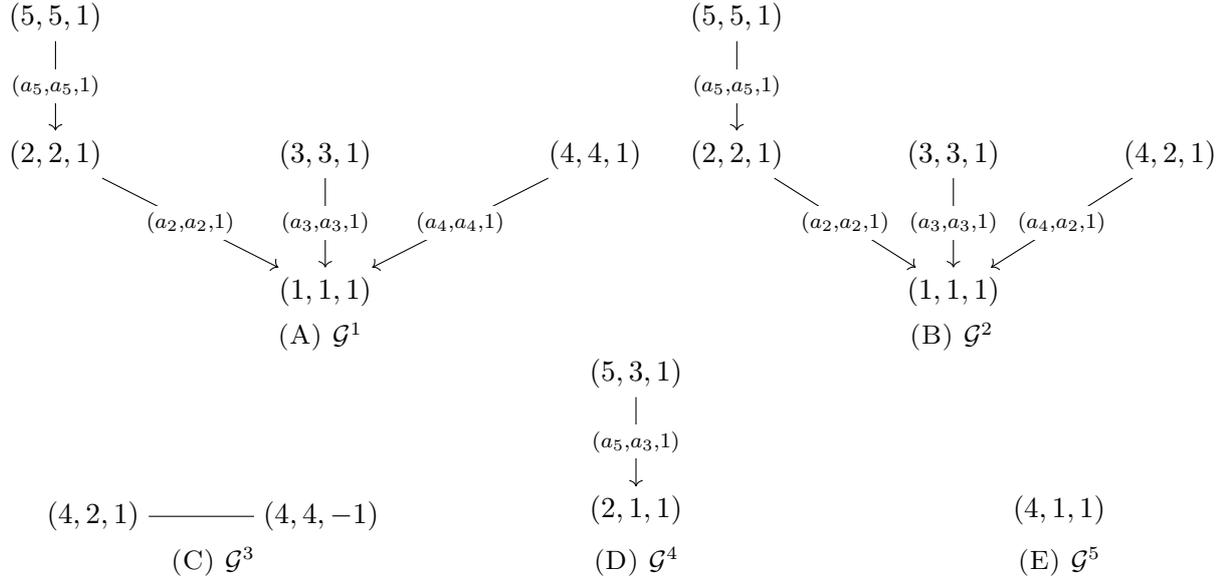
\begin{figure}[H]
        \centering
    \begin{subfigure}[b]{0.499\textwidth}
        \centering
        \begin{tikzcd}[ampersand replacement=\&,cramped,column sep=7mm,row sep=6mm]
	{(5,5,1)} \\
	\\
	{(2,2,1)} \&\&\& {(3,3,1)} \&\&\& {(4,4,1)} \\
	\\
	\&\&\& {(1,1,1)}
	\arrow["{{(a_5,a_5,1)}}"{description}, from=1-1, to=3-1]
	\arrow["{{(a_2,a_2,1)}}"{description}, from=3-1, to=5-4]
	\arrow["{{(a_3,a_3,1)}}"{description}, from=3-4, to=5-4]
	\arrow["{{(a_4,a_4,1)}}"{description}, from=3-7, to=5-4]
\end{tikzcd}
        \caption{$\C G^1$}
        \label{fig: G1}
        \end{subfigure}\hfill
        \begin{subfigure}[b]{0.50\textwidth}
        \centering
            \begin{tikzcd}[ampersand replacement=\&,cramped,column sep=7mm,row sep=6mm]
	{(5,5,1)} \\
	\\
	{(2,2,1)} \&\& {(3,3,1)} \&\& {(4,2,1)} \\
	\\
	\&\& {(1,1,1)}
	\arrow["{{(a_5,a_5,1)}}"{description}, from=1-1, to=3-1]
	\arrow["{{(a_2,a_2,1)}}"{description}, from=3-1, to=5-3]
	\arrow["{{(a_3,a_3,1)}}"{description}, from=3-3, to=5-3]
	\arrow["{{(a_4,a_2,1)}}"{description}, from=3-5, to=5-3]
    \end{tikzcd}
            \caption{$\C G^2$}
            \label{fig:G2}
        \end{subfigure}
        \vfill
        \begin{subfigure}{0.32\textwidth}
        \centering
            \begin{tikzcd}[ampersand replacement=\&,cramped,row sep=6mm,column sep=7mm]
	{(4,2,1)} \&\& {(4,4,-1)}
	\arrow[no head, from=1-3, to=1-1]
\end{tikzcd}
            \caption{$\C G^3$}
            \label{fig: G3}
        \end{subfigure}~
        \begin{subfigure}{0.32\textwidth}
            \[\begin{tikzcd}[ampersand replacement=\&,cramped,row sep=6mm,column sep=7mm]
	{(5,3,1)} \\
	\\
	{(2,1,1)}
	\arrow["{{(a_5,a_3,1)}}"{description}, from=1-1, to=3-1]
\end{tikzcd}\]
        \caption{$\C G^4$}
        \label{fig: G4}
        \end{subfigure}~
        \begin{subfigure}{0.32\textwidth}
            \[\begin{tikzcd}[ampersand replacement=\&,cramped,row sep=6mm,column sep=7mm]
	{(4,1,1 )}
\end{tikzcd}\]
        \caption{$\C G^5$}
        \label{fig: G5}
        \end{subfigure}
        \caption{GGMs from the RTM $M$ (\Cref{exmp: RTM}) to itself (see \Cref{exmp: all GGMs})}
        \label{fig: all GGMs}
    \end{figure}
We end the section with the construction of a quiver morphism between branches of trees given a point in a GGM.
\begin{lem}\label{embedding}
Suppose $\C G$ is a GGM from $M_1$ to $M_2$ and $(n,m)\in \pi(\C G_0)$. Then there exists a quiver morphism $\iota\colon \Br(n)\to \Br(m)$ satisfying $F_2\vert_{\Br(m)}\circ \iota=F_1\vert_{\Br(n)}$.
\end{lem}
\begin{proof}
Suppose $(n,m,l)\in\C G_0$ for some $l\in\{-1,1\}$. We inductively construct a quiver morphism $\iota:\Br(n)\to\Br(m)$ starting with $\iota(n)\coloneqq m$ such that $(n',\iota(n'),l)\in\C G_0$ for each $n'\in\Br(n)_0$.

If $n'\in\Br(n)_0$ satisfies $n'\neq n$ and $(\s p(n'),\iota(\s p(n')),l)\in\C G_0$, then the completeness of $\C G$ at $(\s p(n'),\iota(\s p(n')),l)$ for the incoming arrow $a_{n'}$ in the domain tree (\Cref{defn: completeness}(1)) guarantees the existence of some $m'\in T^2_0$ such that $(n',m',l)\in\C G_0$ and $(a_{n'},b_{m'},l)\in\C G_1$. In particular, $\s p(m')=\iota(\s p(n'))$ and $F_1(a_{n'})=F_2(b_{m'})$. If there is $m''\neq m'$ satisfying the same properties, then $F_2(b_{m'})=F_2(b_{m''})$ and $\s p(m')=\s p(m'')$. Since $((n',m',l)\xrightarrow{(a_{n'},b_{m'},l)}(\s p(n'),\s p(m'),l)\xleftarrow{(a_{n'},b_{m''},l)}(n',m'',l))\in\C R[2]$ and $\C G$ is $\C R[2]$-free, we conclude that $(n',m'',l)\notin\C G_0$. Thus, we can set $\iota(n')\coloneqq m'$ and $\iota(a_{n'})\coloneqq b_{m'}$. This completes the construction of the quiver morphism $\iota$ satisfying the required property.
\end{proof}
The result above together with \Cref{rem:heightpreserve} yields the following consequence.
\begin{cor}\label{cor:heightpreserve}
Suppose $(n,m)\in \pi(\C G_0)$ for some GGM $\C G$ from $M_1$ to $M_2$. Then $\hei{}(\Br(n))\le\hei{}(\Br(m))$.
\end{cor}
\begin{exmp}
Consider the RTM $M$ from \Cref{exmp: RTM}, the GGM $\C G^4$ from $M$ to itself as shown in \Cref{fig: G4}, and $(2,1)\in\pi(\C G^4)$. \Cref{embedding} guarantees the existence of a quiver morphism $\iota\colon \Br(2)\to\Br(1)$--this morphism is defined as $\iota(2)=1$ and $\iota(5)=3$ and $\iota(a_5)=a_3$. It is readily verified that $F\vert_{\Br(1)}\circ\iota=F\vert_{\Br(2)}$.
\end{exmp}

\section{Main results for an RTM with a sink}\label{sec 3}
Let $(T,F)$ be a rooted tree (with a sink) over $(Q,\rho)$ and $M\coloneqq M(T,F)$ be the associated RTM. This section is dedicated to the proof of the main result of the paper. Later we also discuss two applications to decompositions of RTMs with sinks.

Here is the main result as stated in \S~\ref{intro} along with an additional equivalent statement.
\begin{restatable*}{thm}{mainone}(\textbf{Sink version}) \label{thm: main A}
The following are equivalent for the RTM $M$:
    \begin{enumerate}
        \item $M$ is indecomposable;
        \item there is no non-identity idempotent quiver morphism $\iota\colon T\to T$ satisfying $F\circ \iota=F$;
        \item there does not exist a GGM $\C M=(\C M_0,\C M_1,\C E_{\C M})$ with $(n_1,n_2)\in\pi(\C M_0)$ satisfying $n_1\neq n_2$, $\s p(n_1)=\s p(n_2)$ and $F(a_{n_1})=F(a_{n_2})$.
    \end{enumerate}
\end{restatable*}

We will prove the implications $(1)\Rightarrow(2)\Rightarrow(3)\Rightarrow(1).$

Thanks to \Cref{prop: indecomposability condn assem}, the next result is the contrapositive of the implication $(1)\Rightarrow(2)$.
\begin{prop}\label{idemtoidem}
Suppose $T$ is a rooted tree with a sink, $\iota\colon T\to T$ is a non-identity idempotent quiver morphism satisfying $F\circ\iota=F$. Then the linear map $\C I\colon M\to M$ defined on the canonical basis of $M$ as $\C I(v_n)\coloneqq v_{\iota(n)}$ for $n\in T_0$ is an idempotent module homomorphism.
\end{prop}

\begin{proof}
To verify that $\C I$ is a module homomorphism, it is enough to check that $$\C I(\alpha\cdot v_n)=\alpha\cdot v_{\iota(n)} \quad\text{ for }n\in T_0\text{ and }\alpha\in Q_1.$$

Since $\iota$ is a quiver morphism, \Cref{rem:heightpreserve} yields $\iota(*)=*$, and hence we have $\alpha\cdot v_*=\alpha\cdot v_{\iota(*)}=0$ for all $\alpha\in Q_1$. Thus we may assume that $n\in T_0\setminus\{*\}$. Then note that $\alpha\cdot v_n\neq0$ if and only if $\alpha=F(a_n)=F(\iota(a_n))=F(a_{\iota(n)})$ if and only if $\alpha\cdot v_{\iota(n)}\neq0$. Thus, it remains to show that $$\C I(F(a_n)\cdot v_n)=F(a_n)\cdot \C I(v_{\iota(n)}) \quad\text{ for } n\in T_0\setminus\{*\}.$$

Since $F(a_n)\cdot v_n=v_{\s p(n)}$, the left hand side equals $v_{\iota(\s p(n))}$ while the right hand side is $F(a_n)\cdot v_{\iota(n)}=v_{\s p(\iota(n))}$. Since $\iota$ is a quiver morphism, we have $\s p(\iota(n))=\iota(\s p(n))$ as required.

Finally, since $\iota$ is a non-identity idempotent, we see that $\C I$ is a non-trivial idempotent.
\end{proof}

Now we prove the contrapositive of $(2)\Rightarrow(3)$. Suppose $\C G=(\C G_0,\C G_1,\C E_{\C G})$ is a GGM with $(n_1,n_2,l)\in\C G_0$ for some $l\in\{-1,1\}$ satisfying $n_1\neq n_2$, $\s p(n_1)=\s p(n_2)$ and $F(a_{n_1})=F(a_{n_2})$. Since $(n_1,n_2,l)\in\C G_0$, \Cref{embedding} yields a quiver morphism $\iota':\Br(n_1)\to\Br(n_2)$. Since $\s p(n_1)=\s p(n_2)$, it is easily verified that the extension $\iota$ of $\iota'$ by the identity map on $T_0\setminus\Br(n_1)_0$ together with the assignment $a_n\mapsto a_{\iota(n)}$ for $n\in T_0\setminus\{\ast\}$ is a quiver endomorphism of $T$. This endomorphism is idempotent since $\iota(\Br(n_1)_0)\subseteq\Br(n_2)_0$ and the restriction of $\iota$ to $T_0\setminus\Br(n_1)_0$ is identity. Finally, $\iota$ is non-identity follows from $\iota(n_1)=n_2\neq n_1$.

The key step in the proof of the implication $(3)\Rightarrow(1)$ is the following theorem.
\begin{thm}\cite[Theorem~B]{sengupta2025generalised}\label{SK:thm B}
Suppose $M'\coloneqq M(T',F')$ is a generalized tree modules such that the pair $(M',M')$ is ghost-free and the following hypotheses hold:
\begin{enumerate}
    \item[(a)] there is no GGM $\C G$ from $M'$ to itself such that $(n_1,n_2)\in \pi(\C G_0)$ for some $n_1\neq n_2$ satisfying $F(a_{n_1})=F(a_{n_2})$ and $\s p(n_1)=\s p(n_2)$; and
    \item[(b)] if $(n_1,n_2)\in\pi(\C G^1_0)$ for some GGM $\C G^1$ from $M'$ to itself and $n_1\neq n_2$, then $(n_2,n_1)\notin\pi(\C G^2_0)$ for any GGM $\C G^2$ from $M'$ to itself.
\end{enumerate}
Then $M'$ is indecomposable.
\end{thm}

In view of \Cref{thm: ghostfree} and the above theorem, it suffices to prove the next technical result to complete the proof of $(3)\Rightarrow(1)$.

\begin{lem}\label{extra condition}
For the RTM $M$ defined at the beginning of this section, hypothesis $(a)$ of the above theorem implies its hypothesis $(b)$. 
\end{lem}
\begin{proof}
We prove the contrapositive. Since $(b)$ fails, there are $n_1\neq n_2$ in $T_0$ and (possibly equal) GGMs $\C G^1$ and $\C G^2$ such that $(n_1,n_2)\in \pi(\C G^1_0)$ and $(n_2,n_1)\in \pi(\C G^2_0)$. Then \Cref{cor:heightpreserve} yields $\hei{}(\Br(n_1))=\hei{}(\Br(n_2))$. Since $n_1\neq n_2$, neither $n_j$ belongs to the branch of the other. In particular, $*\notin\{n_1,n_2\}$. Replacing $\C G^j$ by $-\C G^j$ if necessary, we may assume that $(n_1,n_2,1)\in\C G^1_0$ and $(n_2,n_1,1)\in\C G^2_0$.

There are two cases depending on whether the completeness of $\C G^1$ (resp. $\C G^2$) at $(n_1,n_2,1)$ (resp. $(n_2,n_1,1)$) for the arrow $a_{n_2}$ (resp. $a_{n_1}$) in the codomain is witnessed by an arrow or an edge.

\noindent \textbf{Case 1:} Suppose the completeness of $\C G^1$ at $(n_1,n_2,1)$ for the arrow $a_{n_2}$ in the codomain is witnessed by an edge, say $\{(n_1,n_2,1),(n_1,n_3,-1)\}\in\C E_{\C G^1})$. Then $\s p(n_2)=\s p(n_3)$ and $F(a_{n_2})=F(a_{n_3})$.

If $n_3=n_1$ then $\C G^1$ witnesses the failure of hypothesis $(a)$.

On the other hand, if $n_1\neq n_3$, then \Cref{embedding} applied to $(n_1,n_3)\in \pi(\C G^1_0)$ (resp. $(n_2,n_1)\in \pi(\C G^2_0)$) yields a quiver morphism $\iota_1:\Br(n_1)\to\Br(n_3)$ (resp. $\iota_2:\Br(n_2)\to\Br(n_1)$) satisfying $F\vert_{\Br(n_3)}\circ \iota_1=F\vert_{\Br(n_1)}$ (resp. $F\vert_{\Br(n_1)}\circ \iota_1=F\vert_{\Br(n_2)}$). Consider the quiver morphism $\iota\coloneqq\iota_1\circ \iota_2\colon\Br(n_2)\to\Br(n_3)$. Then $F\vert_{\Br(n_3)}\circ \iota=F\vert_{\Br(n_2)}$. Now consider the subnetwork $\C G=(\C G_0,\C G_1,\C E_{\C G})$ of $\C N[2]$ defined as follows: 
\begin{align*}
\C G_l&\coloneqq\{(x,x,1)\mid x\in \Br(n_2)_l\}\cup\{(x,\iota(x),-1)\mid x\in\Br(n_2)_l\}\mbox{ for }l\in{0,1};\text{ and}\\ 
\C E_{\C G}&\coloneqq \{\{(n_2,n_2,1),(n_2,n_3,-1)\}\}.
\end{align*}
The subnetwork $\C G$ is clearly $\C R[2]$-free, involution free and connected. The completeness condition at $(n,n,1)\in\C G_0$ (resp. $(n,\iota(n),-1)\in\C G_0$) for an arrow $n'\xrightarrow{a}n$ in the domain is satisfied by $(n',n',1)$ (resp. $(n',\iota(n'),-1)$). Similarly, if $n\neq n_2$, then the completeness condition at $(n,n,1)\in\C G_0$ (resp. $(n,\iota(n),-1)\in\C G_0$) for the outgoing arrow $a_n$ in the codomain is satisfied by $(\s p(n),\s p(n),1)$ (resp. $(\s p(n),\s p(\iota(n)),-1)$). Finally, the completeness at $(n_2,n_2,1)$ (resp. $(n_2,n_3,-1)$) for the arrow $a_{n_2}$ (resp. $a_{n_3}$) in the codomain is satisfied by the unique edge in $\C E_{\C G}$. This completes the proof that $\C G$ is complete. Thus $\C G$ is a GGM witnessing the failure of hypothesis $(a)$ at $(n_2,n_3)\in\pi(G_0)$.

A similar proof works if the completeness of $\C G^2$ at $(n_2,n_1,1)$ for the arrow $a_{n_1}$ in the codomain is witnessed by an edge.

\noindent \textbf{Case 2:} Suppose the completeness of $\C G^1$ at $(n_1,n_2,1)$ for the arrow $a_{n_2}$ in the codomain is witnessed by the arrow $(a_{n_1},a_{n_2},1)$, and the completeness of $\C G^2$ at $(n_2,n_1,1)$ for the arrow $a_{n_1}$ in the codomain is witnessed by the arrow $(a_{n_2},a_{n_1},1)$. Then $(\s p(n_1),\s p(n_2),1)\in\C G^1_0$ and $(\s p(n_2),\s p(n_1),1)\in\C G^2_0$.

If $\s p(n_1)=\s p(n_2)$, then $\C G^1$ witnesses the failure of hypothesis $(a)$ at $(n_1,n_2)\in\pi(\C G^1_0)$.

On the other hand, if $\s p(n_1)\neq\s p(n_2)$, then neither $\s p(n_j)$ belongs to the branch of the other and we can repeat the entire argument for the pairs $(\s p(n_1),\s p(n_2),1)\in\C G^1_0$ and $(\s p(n_2),\s p(n_1),1)\in\C G^2_0$. Since $T$ is finite, the process will terminate by witnessing a failure of hypothesis $(a)$.
\end{proof}

As argued earlier, the proof of the above lemma also completes the proof of \Cref{thm: main A}.

As an application of \Cref{thm: main A}, we provide a way to recursively construct indecomposable RTMs (each with a sink) over a bound quiver algebra.
\begin{cor}\label{cor: 2}
Suppose $T$ is a rooted tree with a sink. Suppose $n_1,\cdots,n_k$ is the list of all the vertices of $T$ such that $\s p(n_j)=\ast$, and that $M_j\coloneqq M(\Br(n_j),F\vert_{\Br(n_j)})$ is indecomposable for each $j$. Then the following are equivalent.
\begin{itemize}
\item[(a)] The module $M$ is decomposable.
\item[(b)] There are distinct $i,j\in\{1,\cdots,k\}$ satisfying $F(a_{n_{i}})=F(a_{n_{j}})$, and a quiver morphism $\iota:\Br(n_i)\to \Br(n_j)$ satisfying $\iota(n_i)=n_j$ and $F\vert_{\Br(n_j)}\circ\iota=F\vert_{\Br(n_i)}$.
\end{itemize}
\end{cor}
\begin{proof}
$((a)\Rightarrow(b))$ Since $M$ is decomposable, \Cref{thm: main A} yields a non-identity idempotent quiver endomorphism $\iota\colon T\to T$ satisfying $F\circ\iota=F$. Thanks to \Cref{rem:heightpreserve}, we have $\iota(\ast)=\ast$. Since $\iota$ is non-identity, there is $m\neq*$ in $T_0$ such that $\iota(m)\neq m$. Let $i,j\in\{1,\cdots,k\}$ be such that $m\in\Br(n_i)$ and $\iota(m)\in\Br(n_j)$. If $i=j$, then $\iota(n_i)=n_i$ and hence $\iota\vert_{\Br(n_i)}:Br(n_i)\to\Br(n_i)$ is a non-identity idempotent quiver endomorphism. Thanks to \Cref{thm: main A}, this contradicts the indecomposability of $M_i$. Hence, $i\neq j$.

Since $\iota(*)=*$, we have $\iota(n_i)=n_j$ and hence $\iota\vert_{\Br(n_i)}:\Br(n_i)\to\Br(n_j)$ is a quiver morphism. Finally, $F\circ\iota=F$ together with $\iota(n_i)=n_j$ implies $F(a_{n_i})=F(\iota(a_{n_i}))=F(a_{\iota(n_i)})=F(a_{n_j})$. Therefore, $(b)$ holds.

$((b)\Rightarrow(a))$ Suppose distinct $i,j\in\{1,\cdots,k\}$ satisfy $F(a_{n_{i}})=F(a_{n_{j}})$, and that there is a quiver morphism $\iota:\Br(n_i)\to \Br(n_j)$ satisfying $\iota(n_i)=n_j$ and $F\vert_{\Br(n_j)}\circ\iota=F\vert_{\Br(n_i)}$. Extend $\iota$ to an endomorphism $\iota'\colon T\to T$ by the identity map on $\ang{T_0\setminus\Br(n_i)_0}$ and $a_{n_i}\mapsto a_{n_j}$. It is readily verified that $\iota'$ is a non-identity idempotent quiver morphism satisfying $F\circ\iota'=F$, and thus $(a)$ follows from \Cref{thm: main A}.
\end{proof}
A special case of this corollary, when $Q$ is a rooted tree with a sink and $\rho=\emptyset$, was proven by Katter and Mahrt in \cite[Lemma~2]{katter_mahrt}.

\newcommand{\fix}[1][\iota]{\mathrm{Fix}(#1)}
Towards the end of this section, we show as an application \Cref{idemtoidem} how a non-identity idempotent quiver endomorphism $\iota\colon T\to T$ satisfying $F\circ\iota=F$ provides a decomposition of $M$. Set $\fix\coloneqq\{n\in T_0\mid \iota(n)=n\}$. Clearly $\ang{\fix}$ is a subtree of $T$ with sink $*$ thanks to \Cref{rem:heightpreserve}. However, $\ang{T_0\setminus\fix}$ need not be connected. Let $T^1,\cdots,T^k$ be the partition of $T_0\setminus\fix$ such that each $\ang{T^j}$ is connected component of $\ang{T_0\setminus\fix}$. The trees $\ang{T^1},\cdots,\ang{T^k}$ are all rooted trees being subtrees of the rooted tree $T$, say with roots $n_1,\cdots,n_k$ respectively. Let $M_j\coloneqq M(\ang{T^j},F\vert_{\ang{T^j}}))$ for $1\leq j\leq k$ and $M_0\coloneqq M(\ang{\fix},F\vert_{\ang{\fix}})$.
\begin{lem}\label{cor: 1}
If $\iota\colon T\to T$ is a non-identity idempotent quiver morphism satisfying $F\circ \iota=F$, then using the notations in the above paragraph, we have $M\cong\bigoplus_{j=0}^k M_j$.
\end{lem}
\begin{proof}
Let $\C I:M\to M,\ \C I(v_n)\coloneqq v_{\iota(n)}$ be the non-trivial idempotent homomorphism obtained from \Cref{idemtoidem}. Consider the idempotent map $\C I'\coloneqq(\boldsymbol{1}_M-\C I):M\to M$. Then we have $$M\cong\ker(\C I')\oplus \im[\C I'].$$

Note that $$\ker(\C I')=\{m\in M\mid(\boldsymbol{1}_M-\C I)(m)=0\}=\mathrm{span}_{\C K}\{v_n\mid n\in\fix\}=M_0,$$ and $$\im[\boldsymbol{1}_M-\C I]=\mathrm{span}_{\C K}\{v_n-v_{\iota(n)}\mid n\in T_0\}=\mathrm{span}_{\C K}\{v_n-v_{\iota(n)}\mid n\in T_0\setminus\fix\}.$$

Define a vector space morphism $\C H:\bigoplus_{j=1}^{k}M_j\to\im[\boldsymbol{1}_M-\C I]$ on the basis elements as $v_n\mapsto (v_n-v_{\iota(n)})$. It is clear that $\C H$ is an isomorphism of vector spaces. To complete the proof that $\C H$ is a module homomorphism, we need to show that $$\C H(\alpha\cdot v_n)=\alpha\cdot(v_n-v_{\iota(n)})\text{ for }n\in T_0\setminus\fix\text{ and }\alpha\in Q_1.$$

We first show that the left side is non-zero if and only if the right side is so too.

On the one hand, injectivity of $\C H$ gives $\C H(\alpha\cdot v_n)\neq 0$ if and only if $\alpha\cdot v_n\neq0$ in $\bigoplus_{j=1}^k M_j$ if and only if $n\in\bigsqcup_{j=1}^k (T^j\setminus\{n_j\})$ and $\alpha=F(a_{n})$ if and only if $\s p(n)\notin\fix$ but $\alpha=F(a_{n})$. On the other hand, $\alpha\cdot(v_n-v_{\iota(n)})\neq0$ if and only if $\alpha\cdot v_n\neq\alpha\cdot v_{\iota(n)}$ in $M$. 

If $\s p(n)\notin\fix$ and $\alpha=F(a_{n})$, then $\alpha\cdot v_n=v_{\s p(n)}$, $v_{\s p(n)}\neq v_{\s p(\iota(n))}$ and $\alpha\cdot v_{\iota(n)}\in\{0,v_{\s p(\iota(n))}\}$ so that $\alpha\cdot v_n\neq\alpha\cdot v_{\iota(n)}$ in $M$. 

If $\s p(n)\in\fix$ then $\s p(n)=\s p(\iota(n))$ and $F(a_n)=F(a_{\iota(n)})$. Thus, if $\s p(n)\in\fix$ and $\alpha=F(a_n)$, then $\alpha\cdot v_n=v_{\s p(n)}=v_{\s p(\iota(n))}=\alpha\cdot v_{\iota(n)}$ in $M$. Also since $F(a_n)=F(\iota(a_n))=F(a_{\iota(n)})$, if $\s p(n)\notin\fix$ but $\alpha\neq F(a_n)$, then $\alpha\cdot v_n=0=\alpha\cdot v_{\iota(n)}$.

Therefore, it only remains to show that $$\C H(F(a_n)\cdot v_n)=F(a_n)\cdot(v_n-v_{\iota(n)})\text{ if }\s p(n)\notin\fix.$$ Indeed, we have $\C H(F(a_n)\cdot v_n)=\C H(v_{\s p(n)})=v_{\s p(n)}-v_{\iota(\s p(n))}=F(a_n)\cdot v_n-F(a_{\iota(n)})\cdot v_{\iota(n)}=F(a_n)\cdot (v_n-v_{\iota(n)})$, where the last equality uses $F(a_n)=F(\iota(a_n))=F(a_{\iota(n)})$.
\end{proof}
This decomposition result could be used iteratively for $\ang{\fix}$ and each $\ang{T^j}$ by finding non-identity idempotent quiver endomorphisms of $\fix$ and $T^j$ respectively until no such non-identity idempotents exist.
\begin{exmp}
Recall that for the RTM $M$ from \Cref{exmp: RTM}, we described all the GGMs from $M$ to itself in \Cref{exmp: all GGMs}. Note that $(4,2,1)\in\C G^2$ satisfies $\s p(4)=\s p(2)=1$ and $F(a_4)=F(a_2)$, where $\C G^2$ is shown in \Cref{fig:G2}. Then by violation of \Cref{thm: main A}(3), we conclude that $M$ is decomposable. The associated non-identity idempotent quiver morphism $\iota\colon T\to T$ guaranteed by the violation of \Cref{thm: main A}(1) is given by $\iota(1)=1,\iota(2)=2,\iota(3)=3,\iota(4)=2,\iota(5)=5$ and the $\iota(a_n)=a_{\iota(n)}$ for all $n\in T_0\setminus\{1\}$. It is readily verified that $\iota$ is a non-identity idempotent quiver morphism  satisfying $F\circ\iota=F$. Therefore, \Cref{cor: 1} yields a decomposition of the RTM $M$ as $M_0\oplus M_1$, where $M_0$ is the RTM associated with the restriction of $F$ to $\fix$, which is the subtree shown in red in \Cref{fig: decomposition} and $M_1$ is the simple RTM associated with the singleton tree with vertex $4$. Note that $\fix$ does not admit any non-identity idempotent quiver endomorphism, and hence is indecomposable thanks to \Cref{thm: main A}. As a result, $M\cong M_0\oplus M_1$ is a decomposition of $M$ into indecomposables.
\end{exmp}
    \begin{figure}[H]
        \centering
        \[\begin{tikzcd}[ampersand replacement=\&,cramped]
	\textcolor{rgb,255:red,255;green,35;blue,15}{5} \\
	\textcolor{rgb,255:red,255;green,35;blue,15}{2} \& \textcolor{rgb,255:red,255;green,35;blue,15}{3} \& 4 \\
	\& \textcolor{rgb,255:red,255;green,35;blue,15}{1}
	\arrow["{{a_5}}"{description}, color={rgb,255:red,255;green,35;blue,15}, from=1-1, to=2-1]
	\arrow["{{a_2}}"{description}, color={rgb,255:red,255;green,35;blue,15}, from=2-1, to=3-2]
	\arrow["{{a_3}}"{description}, color={rgb,255:red,255;green,35;blue,15}, from=2-2, to=3-2]
	\arrow["{a_4}"{description}, dotted, from=2-3, to=3-2]
\end{tikzcd}\]
        \caption{The $\fix$ is shown in red color}
        \label{fig: decomposition}
    \end{figure}
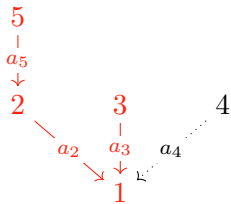

\section{Results for RTMs with a source}\label{sec 4}
The goal of this section is to state and prove the duals of all the important results from the previous two sections (sink case) for rooted trees with a source. Even though most definitions, statements and proofs could be obtained by reversing all the arrows in the sink case, we highlight the major differences in definitions and proofs. In this section, the notation $(T,*)$ (possibly with decoration) will denote a rooted tree with a source $*$, unless otherwise stated.

Let $(T^j,*_j)$ be rooted trees, $F_j:T^j\to(Q,\rho)$ be quiver morphisms and $M_j\coloneqq M(T^j,F_j)$ for $j=1,2$.

The first major change is in the definition of the pullback network associated with the pair $(M_1,M_2)$, where $\C E^1\coloneqq\{\{(n,m),(n',m)\}\mid \s p(n)=\s p(n'), F_1(a_n)=F_1(a_{n'})\}$. As a result, there are only two types of triangles in $\C N[1]$ as shown in \Cref{triatypeop}.
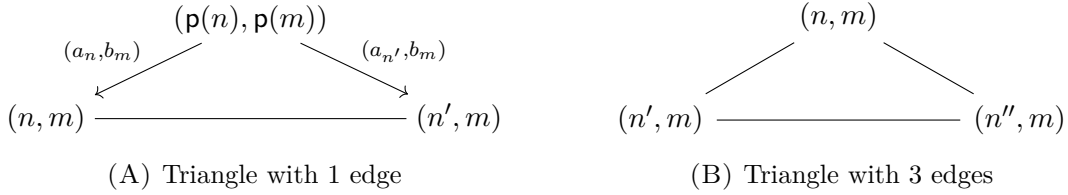
\begin{figure}[H]
    \centering
    \begin{subfigure}[b]{0.45\textwidth}
    \centering
        \[\begin{tikzcd}[ampersand replacement=\&,cramped]
	\& {(\s p(n),\s p(m))} \\
	{(n,m)} \&\& {(n',m)}
	\arrow["{{(a_n,b_{m})}}"', to=2-1, from=1-2]
	\arrow[no head, from=2-1, to=2-3]
	\arrow["{{(a_{n'},b_m)}}", to=2-3, from=1-2]
\end{tikzcd}\]
        \caption{Triangle with $1$ edge}
        \label{arrtriaop}
    \end{subfigure}%
    ~
    \begin{subfigure}[b]{0.45\textwidth}
    \centering
        \begin{tikzcd}
           & (n,m) \\
            (n',m)& & (n'',m)
            \arrow[no head,from=2-1,to=1-2]
            \arrow[no head,from=2-3,to=1-2]
            \arrow[no head, from=2-1,to=2-3]
        \end{tikzcd}
        \caption{Triangle with $3$ edges}
        \label{noarrtriaop}
    \end{subfigure}
    \caption{Triangles in $\C N[1]$}
    \label{triatypeop}
\end{figure}

Accordingly, the definition of completeness changes as follows.
\begin{defn}\label{defn: completenessop}
A subnetwork $\C M=(\C M_0,\C M_1,\C E_{\C M})$ of $\C N[2]$ is said to be \emph{complete} if for each $(n,m,j)\in \C M_0$, we have the following:
\begin{enumerate}
    \item If $(n'\xrightarrow{a_{n}}n)$ in $T^1$, then at least one of the following holds:
    \begin{enumerate}
        \item there exists $(m'\xrightarrow{b_m}m)$ such that $((n',m',j)\xrightarrow{(a_n,b_m,j)}(n,m,j))\in \C M_1$;
        \item there exists $(n\xleftarrow{a_n}n'\xrightarrow{a_{n''}}n'')$ with $F_1(a_n)=F_1(a_{n''})$ and $(n'',m,-j)\in \C M_0$ such that $$\{(n,m,j),(n'',m,-j)\}\in \C E_{\C M}.$$
    \end{enumerate}
    \item If $(m\xrightarrow{b_{m'}}m')$, then there exists $(n\xrightarrow{a_{n'}}n')$ in $T^1$ such that $$((n,m,j)\xrightarrow{(a_{n'},b_{m'},j)}(n',m',j))\in \C M_1.$$
\end{enumerate}
\end{defn}

With this new definition, we get the following analogue of \Cref{GGMgeneration}.
\begin{prop}\label{thm: ghostfreeop}
If $(M_1,M_2)$ is a pair of RTMs with sources, then it is ghost-free, and hence $\Hom$ is the $\C K$-span of the set $\{\C H_{\C G}\mid\C G\text{ is a GGM from }M_1\text{ to }M_2\}$.
\end{prop}
The proof of the above result is the combination of duals of the proofs of \Cref{thm: ghostfree} and \Cref{GGMgeneration}, and hence omitted.

\begin{rmk}\label{ghostfreegeneral}
The proof of \Cref{thm: ghostfreeop} could be adapted to obtain the same conclusion (with appropriate modifications in the definitions as per \cite{sengupta2025generalised}) when $T^1$ is a rooted tree with a sink and $T^2$ is a rooted tree with a source.
\end{rmk}
We expect that the following question has an affirmative answer.
\begin{que}
Does the conclusion of \Cref{thm: ghostfreeop} hold when $T^1$ is a rooted tree with a source and $T^2$ is a rooted tree with a sink?
\end{que}

The main result in this section is the following with exactly the same statement as the sink version.
\begin{restatable*}{thm}{mainone}(Source version)\label{thm: main A op}
Let $(T,F)$ be a rooted tree (with a source) over $(Q,\rho)$. Then the following are equivalent for the RTM $M\coloneqq M(T,F)$:
\begin{enumerate}
    \item $M$ is indecomposable;
    \item there is no non-identity quiver morphism $\iota\colon T\to T$ satisfying $F\circ \iota=F$;
    \item there does not exist a GGM $\C M=(\C M_0,\C M_1,\C E_{\C M})$ with $(n_1,n_2)\in \pi(\C M)_0$ satisfying $n_1\neq n_2$, $\s p(n_1)=\s p(n_2)$ and $F(a_{n_1})=F(a_{n_2})$. 
\end{enumerate}
\end{restatable*}

The proofs of the implications $(2)\Rightarrow(3)\Rightarrow(1)$ are dual to the sink case. Thanks to \Cref{prop: indecomposability condn assem}, the proof of the contrapositive of $(1)\Rightarrow(2)$, and hence that of the source version of \Cref{thm: main A} is complete once the following proposition is established.
\begin{prop}\label{idemtoidemop}
Suppose $T$ is a rooted tree with a source, $\iota:T\to T$ is a non-identity idempotent quiver morphism satisfying $F\circ\iota=F$. Then the linear map $\C I\colon M\to M$ on the canonical basis of $M$ as $\C I(v_n)\coloneqq\sum_{m\in\iota^{-1}(n)}v_m$ for $n\in T_0$ is a non-trivial idempotent module homomorphism.
\end{prop}

Before proving the above proposition, we note a key property of $\iota$.
\begin{rmk}\label{inverseidem}
If $\iota\colon T\to T$ is idempotent, then $\iota^{-1}(\iota^{-1}(n))=\iota^{-1}(n)$ for each $n\in T_0$.
\end{rmk}

\begin{proof}
To verify that $\C I$ is a module homomorphism we need to check that $$\C I(\alpha\cdot v_n)=\alpha\cdot\C I(v_n)\quad \text{ for all }n\in T_0\text{ and }\alpha\in Q_1.$$

If $m\in\iota^{-1}(n)$, then $\alpha\cdot v_m\neq0$ if and only if $\alpha=F(a_{m'})=F(\iota(a_{m'}))=F(a_{\iota(m')})$ for some $m'\in T_0\setminus\{*\}$ with $\s p(m')=m$ if and only if $\alpha=F(a_{n'})$ for some $n'\in T_0\setminus\{*\}$ with $\s p(n')=n$ if and only if $\alpha\cdot v_n\neq 0$. Thus, $\alpha\cdot v_n=0$ if and only if $\alpha\cdot\C I(v_n)=\alpha\cdot(\sum_{m\in\iota^{-1}(n)}v_m)=0$. Therefore, it is sufficient to show that
$$\C I(F(a_n)\cdot v_{\s p(n)})=F(a_n)\cdot\C I(v_{\s p(n)})\quad \text{ for all }n\in T_0\setminus\{\ast\}.$$

Set $A_{a_n,x}\coloneqq\{m\in T_0 \mid F(a_m)=F(a_n)\text{ and }\s p(m)=x\}$ for $n,x\in T_0$ so that $F(a_n)\cdot v_{\s p(n)}=\sum_{j\in A_{a_n,\s p(n)}}v_j$. Then $$\C I(F(a_n)\cdot v_{\s p(n)})=\sum_{j\in A_{a_n,\s p(n)}}\left(\sum_{m\in \iota^{-1}(j)}v_m\right)$$ and $$F(a_n)\cdot\C I(v_{\s p(n)})=F(a_n)\cdot\left(\sum_{m\in \iota^{-1}(\s p(n))}v_m\right)=\sum_{m\in \iota^{-1}(\s p(n))}\left(\sum_{j\in A_{a_n,m}}v_j\right).$$ To show that the above two expressions are equal, we need to show the equality of the index sets $X\coloneqq\{m\in \iota^{-1}(j)\mid j\in A_{a_n,\s p(n)}\}$ and $Y\coloneqq\{j\in A_{a_n,m}\mid m\in \iota^{-1}(\s p(n))\}$--here we use the observation that $A_{a_n,m_1}\cap A_{a_n,m_2}=\emptyset$ whenever $m_1\neq m_2$.
    
If $x\in X$, then $x\in \iota^{-1}(j')$ for some $j'\in A_{a_n,\s p(n)}$. Hence, $\iota(x)\in A_{a_n,\s p(n)}$. Using $F\circ\iota=F$ and the fact $\iota$ is a quiver morphism, we get $F(a_n)=F(a_{\iota(x)})=F(\iota(a_x))=F(a_x)$ and $\s p(n)=\s p(\iota(x))=\iota(\s p(x))$. Thus, $x\in A_{a_n,\s p(x)}$, and hence $x\in Y$. Therefore, $X\subseteq Y$.

If $y\in Y$, then $y\in A_{a_n,m}$ for some $m$ satisfying $\iota(m)=\s p(n)$. By the definition of $A_{a_n,m}$, we have $F(a_y)=F(a_n)$ and $m=\s p(y)$. Hence, $\s p(\iota(y))=\iota(\s p(y))=\s p(n)$. Thus, we have $\s p(\iota(y))=\s p(n)$ and $F(a_{\iota(y)})=F(\iota(a_y))=F(a_y)=F(a_n)$, which shows that $\iota(y)\in A_{a_n,\s p(n)}$. This completes the proof that $y\in X$, and hence that of $Y\subseteq X$.

The last two paragraphs show that $\C I$ is a module homomorphism. Moreover, since $\iota$ is idempotent, for any $n\in T_0$, in view of \Cref{inverseidem}, we have $$\C I^2(v_n)=\C I\left(\sum_{m\in \iota^{-1}(n)}v_m\right)=\sum_{m\in \iota^{-1}(n)}\left(\sum_{j\in \iota^{-1}(m)}v_j\right)=\sum_{j\in \iota^{-1}(\iota^{-1}(n))}v_j=\sum_{j\in \iota^{-1}(n)}v_j=\C I(v_n).$$ Hence $\C I$ is idempotent. The non-triviality of $\C I$ is immediate from the fact that $\iota$ is non-identity.
\end{proof}

As an application of the source version of \Cref{thm: main A} we provide a method to recursively construct indecomposable RTMs with sources--the statement as well as the proof is dual to those of \Cref{cor: 2} and hence omitted.
\begin{cor}\label{cor: 2op}
Suppose $T$ is a rooted tree with a source. Suppose $n_1,\cdots,n_k$ is the list of all the vertices of $T$ such that $\s p(n_j)=\ast$, and that $M_j\coloneqq M(\Br(n_j),F\vert_{\Br(n_j)})$ is indecomposable for each $j$. Then the following are equivalent.
\begin{itemize}
\item[(a)] The module $M$ is decomposable.
\item[(b)] There are distinct $i,j\in\{1,\cdots,k\}$ satisfying $F(a_{n_{i}})=F(a_{n_{j}})$, and a quiver morphism $\iota:\Br(n_i)\to \Br(n_j)$ satisfying $\iota(n_i)=n_j$ and $F\vert_{\Br(n_j)}\circ\iota=F\vert_{\Br(n_i)}$.
\end{itemize}
\end{cor}
\begin{exmp}\label{exmp: source}
As an application to \Cref{cor: 2op}, we construct an indecomposable RTM over the bound quiver $(L_2,\rho)$, where $L_2$ is shown in \Cref{fig:placeholder}(A) and $\rho$ is the set of all length three paths in $L_2$. Consider the rooted tree $(T,1)$ as shown in \Cref{fig:placeholder}(B). Suppose $F:T\to(L_2,\rho)$ satisfies either $F(a_4)\neq F(a_5)$ or $F(a_2)\neq F(a_3)$. Let $M_j\coloneqq M(\Br(j),F\vert_{\Br(j)})$. First note that both $M_4$ and $M_5$ are simple, and hence indecomposable. Indecomposability of $M_2$ and $M_3$ is clear as they are tree modules. Finally, the assumption on $F$ ensures that Condition $(b)$ of \Cref{cor: 2op} fails, which is equivalent to the indecomposability of $M=M_1$.
\end{exmp}
    \begin{figure}[H]
        \centering
        \begin{subfigure}[b]{0.4\textwidth}
        \centering
            \[\begin{tikzcd}[ampersand replacement=\&,cramped]
	{\boldsymbol{1}}
	\arrow["\alpha", from=1-1, to=1-1, loop, in=55, out=125, distance=10mm]
	\arrow["\beta", from=1-1, to=1-1, loop, in=235, out=315, distance=10mm]
\end{tikzcd}\]
        \caption{Quiver $L_2$}
        \end{subfigure}~
        \begin{subfigure}[b]{0.5\textwidth}
            \[\begin{tikzcd}[ampersand replacement=\&,cramped]
	\& 1 \\
	2 \&\& 3 \\
	4 \&\& 5
	\arrow["{a_2}"', from=1-2, to=2-1]
	\arrow["{a_3}", from=1-2, to=2-3]
	\arrow["{a_4}"', from=2-1, to=3-1]
	\arrow["{a_5}", from=2-3, to=3-3]
\end{tikzcd}\]
    \subcaption{Rooted tree $T$ with source $1$}
        \end{subfigure}
        \caption{Constructing an indecomposable RTM in \Cref{exmp: source}}
        \label{fig:placeholder}
    \end{figure}
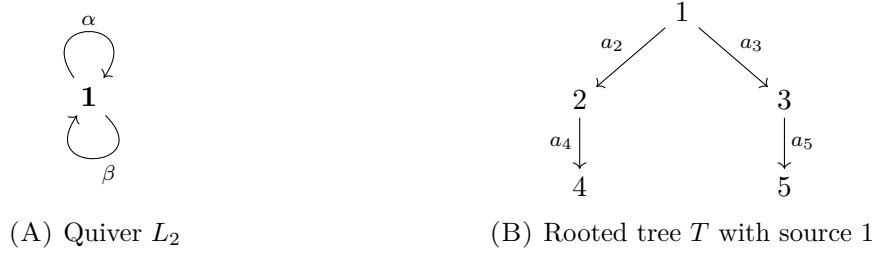

At the end, we state and prove a result to decompose RTMs with a source using the data of a non-identity idempotent quiver morphism--this results is the dual of \Cref{cor: 1}--as an application of \Cref{idemtoidemop}. Let $\iota\colon T\to T$ be an idempotent quiver endomorphism satisfying $F\circ\iota=F$. Set $\im\coloneqq\{\iota(n)\mid n\in T_0\}$. Then $\ang{\im}$ is a rooted subtree of $T$ with source $\ast$. Set $M_0\coloneqq M(\ang{\im},F\vert_{\ang{\im}})$. However, $\ang{T_0\setminus \im}$ need not be connected. Let $T^1,\cdots,T^k$ be a partition of $T_0\setminus\im$ such that each $\ang{T^j}$ is a connected component of $\ang{T_0\setminus\im}$, and hence a rooted tree, say with root $n_j$. Set $M_j\coloneqq M(\ang{T^j},F\vert_{\ang{T^j}})$ for $1\leq j\leq k$.
\begin{lem}\label{cor: 1op}
If $\iota\colon T\to T$ is a non-identity idempotent quiver morphism satisfying $F\circ \iota=F$, then using the notations in the above paragraph, $M\cong\bigoplus_{j=0}^k M_j.$
\end{lem}
\begin{proof}
Let $\C I:M\to M$ be the non-trivial idempotent module homomorphism from \Cref{idemtoidemop} defined on basis vectors by $v_n\mapsto\sum_{m\in \iota^{-1}(n)}v_{m}$. Then we have $$M\cong\ker(\C I)\oplus \im[\C I].$$

Note that $$\ker(\C I)=\mathrm{span}_{\C K}\{v_n\mid\sum_{m\in \iota^{-1}(n)}v_m=0\}=\mathrm{span}_{\C K}\{v_n\mid \iota^{-1}(n)=\emptyset\}=\mathrm{span}_{\C K}\{v_n\mid n\notin\im\}\cong\bigoplus_{j=1}^k M_j,$$ and
$$\im[\C I]=\{\C I(m)\mid m\in M\}=\mathrm{span}_{\C K}\{\C I(v_n)\mid n\in T_0\}=\mathrm{span}_{\C K}\{\C I(v_n)\mid n\in\im\},$$ where the last equality follows from $\C I(v_n)\neq 0\text{ if and only if }\iota^{-1}(n)\neq \emptyset\text{ if and only if } n\in \im$.

We need to show that the vector space isomorphism $\C I\vert_{M_0}\colon M_0\to\im[\C I]$ is a module isomorphism, i.e., $$\C I\vert_{M_0}(\alpha\cdot v_n)=\alpha\cdot \C I\vert_{M_0}(v_n) \text{ for } n\in \im_0\text{ and } \alpha\in Q_1.$$

On the one hand, since $\C I\vert_{M_0}$ is injective, $\C I\vert_{M_0}(\alpha\cdot v_n)\neq 0$ in $\im[\C I]$ if and only if $\alpha\cdot v_n\neq 0$ in $M_0$ if and only if $\alpha=F(a_j)$ for some $j\in\im$ satisfying $\s p(j)=n$.

On the other hand, $\alpha\cdot\C I\vert_{M_0}(v_n)=\alpha\cdot \C I(v_n)=\alpha\cdot (\sum_{m\in \iota^{-1}(n)}v_m)=\sum_{m\in \iota^{-1}(n)}(\alpha\cdot v_m)$. Hence, $\alpha\cdot\C I\vert_{M_0}(v_n)\neq 0$ if and only if there is $m\in\iota^{-1}(n)$ such that $\alpha\cdot v_m\neq 0$ in $M$ if and only if there is some $j'\in T_0$ such that $\alpha=F(a_{j'})=F(\iota(a_{j'}))=F(a_{\iota(j')})$, $m=\s p(j')$, and hence $n=\iota(m)=\iota(\s p(j'))=\s p(\iota(j'))$.

The previous two paragraphs together imply that it is sufficient to show that 
$$\C I\vert_{M_0}(F(a_n)\cdot v_{\s p(n)})=F(a_n)\cdot\C I\vert_{M_0}(v_{\s p(n)})\text{ for }n\in \im\setminus\{\ast\}.$$ 
We have \begin{align*}\C I\vert_{M_0}(F(a_n)\cdot v_{\s p(n)})&=\C I\vert_{M_0}
\left(\sum_{m\in\im\cap A_{a_n,\s p(n)}}v_m\right)=\sum_{m\in\im\cap A_{a_n,\s p(n)}}\C I(v_m)\\&=\sum_{m\in\im\cap A_{a_n,\s p(n)}}\left(\sum_{m'\in\iota^{-1}(m)}v_{m'}\right)\end{align*} and \begin{align*}
    F(a_n)\cdot \C I\vert_{M_0}(v_{\s p(n)})&=F(a_n)\cdot \C I(v_{\s p(n)})=F(a_n)\cdot\left(\sum_{j\in\iota^{-1}(\s p(n))}v_j\right)\\&=\sum_{j\in\iota^{-1}(\s p(n))}(F(a_n)\cdot v_j)=\sum_{j\in\iota^{-1}(\s p(n))}\sum_{j'\in A_{a_n,j}}v_{j'}.
\end{align*}
To complete the proof, we need to show that equality between the index sets $X\coloneqq\{m'\mid\iota(m')\in A_{a_n,\s p(n)}\}$ and $Y\coloneqq\bigsqcup_{\iota(j)=\s p(n)}A_{a_n,j}$.

If $x\in X$, then $\iota(x)\in A_{a_n,\s p(n)}$. Hence, $F(a_n)=F(a_{\iota(x)})=F(\iota(a_x))=F(a_x)$ and $\s p(n)=\s p(\iota(x))=\iota(\s p(x))$. Therefore, $x\in A_{a_n,\s p(x)}$ and $\iota(\s p(x))=\s p(n)$ so that $x\in Y$.

If $y\in Y$, then $y\in A_{a_n,j}$ for some $j$ satisfying $\iota(j)=\s p(n)$. Then $F(a_n)=F(a_y)=F(\iota(a_y))=F(a_{\iota(y)})$ and $\s p(y)=j$, so that $\iota(\s p(y))=\iota(j)=\s p(n)$. Therefore, $\iota(y)\in A_{a_n,\s p(n)}$ so that $y\in X$. This completes the proof. 
\end{proof}

The above corollary could be used to recursively decompose a RTM with a source into indecomposable direct summands.
\bibliographystyle{alpha}
\bibliography{main}
\end{document}